\documentclass[11pt]{amsart}
\usepackage{fullpage}   % use as much space as possible
\usepackage{amssymb}    % AMS fonts
\usepackage{amsmath}
\usepackage{amsthm}

\theoremstyle{plain}
\newtheorem{theorem}{Theorem}[section]
\newtheorem{lemma}[theorem]{Lemma}

\newtheorem{proposition}[theorem]{Proposition}
\newtheorem{corollary}[theorem]{Corollary}
\newtheorem{conjecture}[theorem]{Conjecture}

\newtheorem{example}[theorem]{Example}
\newtheorem{question}[theorem]{Question}
\theoremstyle{definition}
\newtheorem*{acknowledgment}{Acknowledgment}
\newtheorem*{remark}{Remark}

\numberwithin{equation}{section}

\newcommand{\Inn}{\mathrm{Inn}}
\newcommand{\Mlt}{\mathrm{Mlt}}
\newcommand{\Aut}{\mathrm{Aut}}

\newcommand{\ldiv}{\backslash}
\newcommand{\rdiv}{/}

\newcommand{\Atp}{\mathrm{Atp}}

\title{Simple right conjugacy closed loops}
\author{Mark Greer}
\email{mgreer@una.edu}
\address{Department of Mathematics,
One Harrison Plaza, University of North Alabama,
Florence, AL 35632 USA}

\begin{document}
\begin{abstract}
We give a general construction for right conjugacy closed loops, using $GL(2,q)$ for $q$ a prime power.  Under certain conditions, the loops constructed are simple, giving the first general construction for finite, simple right conjugacy closed loops.  We give a complete description of the isomorphism classes for the construction, yielding an exact count of non isomorphic loops for each $q$.  
\end{abstract}

\maketitle
\allowdisplaybreaks

\section{Introduction}
\label{S1}
A \emph{loop} $(Q,\cdot)$ consists of a set $Q$ with a binary operation $\cdot : Q\times Q\to Q$ such that (i) for all $a,b\in Q$, the equations $ax = b$ and $ya = b$ have unique solutions $x,y\in Q$, and (ii) there exists $1\in Q$ such that $1x = x1 = x$ for all $x\in Q$. We denote these unique solutions by $x = a\ldiv b$ and $y = b\rdiv a$, respectively.  Standard references in loop theory are \cite{bruck71, pflugfelder90}.

We say a subset $S$ of a group $G$ is closed under conjugation if $x^{-1}yx\in S$ for all $x,y\in S$.  A loop $Q$ is a \emph{right conjugacy closed loop} (or RCC loop) if $R_{Q}$ is closed under conjugation.  That is, $R_{x}^{-1}R_{y}R_{x}\in R_{Q}$ for all $x,y\in Q$.  Similarly, a loop $Q$ is \emph{left conjugacy closed} (LCC) if $L_{x}^{-1}L_{y}L_{x}\in L_{Q}$ for all $x,y\in Q$.  Most of the literature on the one-sided conjugacy closed loops deals with left conjugacy closed loops \cite{basarab91, drapal03, drapal04, NS94}.  RCC loops are the more natural choice here since our permutations act on the right.

For (two-sided) CC-loops, the existence of nonassociative simple loops is settled in the negative by Basarab's Theorem \cite{basarab91}: \emph{The factor of a CC-loop by its (necessarily normal) nucleus is an abelian group}. It follows that a simple CC-loop must have nucleus coinciding with the whole loop, hence is a group.

In the one-sided case, nonassociative simple RCC loops are known to exist. The first example occurring in the literature seems to be the simple Bol loop of exponent $2$ and order $96$ constructed by G. Nagy \cite{nagy09}, because a right Bol loop of exponent $2$ is necessarily an RCC loop. Other examples arose in the computer search for nonassociative, finite simple automorphic loops \cite{JKNV11}, since every RCC loops is a right automorphic loop.

Here we give the first general construction of a large class of nonassociative, finite simple RCC loops. Our construction by no means accounts for all such loops; for example, Nagy's Bol loop of exponent $2$ does not fit this construction. Thus a full classification of finite simple RCC loops is still elusive. Nevertheless, using \textsc{Mace4} \cite{mccune09} and the loops package for GAP \cite{GAP,GAPNV},  we have found by exhaustive computer search that our construction accounts for all finite simple RCC loops up to order $15$.

It turns out that our construction is isomorphic to a construction due to Hall for projective planes, called \emph{Hall Planes} \cite{hall59}.  For each irreducible monic quadratic over a finite field, there is a unique Hall quasifield created.  Hall planes of the same order turn out to be isomorphic \cite{NSG86}.  The multiplication loops in Hall's quasifields are isomorphic to the RCC loops constructed in this paper.  However, we will have non-trivial isomorphism classes (Theorem \ref{isocount}).  It is interesting to note another paper which relates quasifields to (one-sided) conjugacy closed loops is \cite{matievics97}.

In section $\S$2, we review basic notions from loop theory.  We also give some basic results concerning with RCC loops.  In section $\S$3, we give our construction for our loops (Theorem \ref{construction}) and prove these loops are indeed RCC (Theorem \ref{RCC}).  After proving several structural results about these RCC loops, we show our construction gives simple RCC loops (Theorem \ref{simple}).

As previously stated, our construction may give isomorphic RCC loops for certain orders.  We explain this phenomenon in $\S$4, showing that isomorphisms reduce to automorphisms of finite fields (Theorem \ref{isoauto}) and count the isomorphism classes (Theorem \ref{isocount}).  Finally, we end with some open questions.
% % % % % % % % % % % % % % % % % % % % % % % % % % % % % % % % % % % % % % % % % % %
% % % % % % % % % % % % % % % % % % % % % % % % % % % % % % % % % % % % % % % % % % %
% % % % % % % % % % % % % % % % % % % % % % % % % % % % % % % % % % % % % % % % % % %
\section{Right Conjugacy Closed loops}
\label{S2}
To avoid excessive parentheses, we use the following convention:
\begin{itemize}
\item multiplication $\cdot$ will be less binding than divisions $\backslash, /$.
\item divisions are less binding than juxtaposition
\end{itemize}
For example $xy/z \cdot y\backslash xy$ reads as $((xy)/z)(y\backslash (xy))$.  

In a loop \emph{Q}, the left and right translations by $x \in Q$ are defined by $yL_{x} = xy$ and $yR_{x} = yx$ respectively.  We thus have $\backslash,/$ as $x\backslash y=yL_{x}^{-1}$ and $y/x=yR_{x}^{-1}$.  We define the \emph{right section} of $Q$, $R_{Q}=\{R_{x}\mid x\in Q\}$, \emph{right multiplication group} of $Q$, $\Mlt_{\rho}(Q)= \left\langle R_{x}\mid x\in Q\right\rangle$ and \emph{multiplication group} of $Q$, $\Mlt(Q)=\left\langle R_{x},L_{x}\mid x\in Q\right\rangle$.  We define the \emph{inner mapping group} of $Q$, $\Inn(Q)=\Mlt(Q)_{1}= \{\theta\in \Mlt(Q) \mid 1\theta=1\}$.

A bijection $\theta:Q\rightarrow Q$ is an \emph{automorphism} if $(xy)\theta=x\theta \cdot y\theta$ for all $x,y\in Q$.  Under composition, we define the \emph{automorphism group of Q} $\Aut(Q)$.  A triple $(\alpha,\beta,\gamma)$ of bijections of a loop $Q$ is an \emph{autotopism} if for all $x,y\in Q$, $x\alpha\cdot y\beta = (xy)\gamma$. The set $\Atp(Q)$ of all autotopisms of $Q$ is a group under composition.

A subloop $N$ of $Q$ is normal ($N\trianglelefteq Q$) if for every $x, y\in Q$ we have $xN=Nx,$ $x  (yN) = (xy)N$, $(Nx) y = N(xy)$, and $x(Ny) = (xN)y$.  A loop $Q$ is simple if the only normal subloops are the trivial, $\{1\}$ and $Q$.

The following proposition, which will be useful later as it uses right translations to check whether a magma $(Q,\cdot)$ is indeed a loop.
\begin{proposition}(\cite{JKNV11})
Let $(Q,\cdot)$ be a magma with $1\in Q$ an identity element.  Then $Q$ is a loop \emph{if and only if} $R_{x}R_{y}^{-1}$ is fixed point free for every $x,y\in Q$ with $x\neq y$ and $x,y\neq 1$.
\label{rightsectionloop}
\end{proposition}
\begin{proof}
This follows from Lemmas $2.1$ and $2.2$.
\end{proof}

For a loop $Q$, we have the following subsets of interest:
\begin{center}
\begin{tabular}{ll}
\emph{the left nucleus of Q}, &$N_{\lambda}(Q)=\{a\in Q\mid a\cdot xy=ax\cdot y \ \forall x,y\in Q\ \}$,\\
\emph{the middle nucleus of Q}, &$N_{\mu}(Q)=\{a\in Q\mid x\cdot ay=xa\cdot y \ \forall x,y\in Q\ \}$,\\
\emph{the right nucleus of Q}, &$N_{\rho}(Q)=\{a\in Q\mid x\cdot ya=xy\cdot a  \ \forall x,y\in Q\ \}$,\\
\emph{the nucleus of Q}, &$N(Q)=N_{\lambda}(Q)\cap N_{\mu}(Q)\cap N_{\rho}(Q)$,\\
\emph{the commutant of Q}, &$C(Q)=\{a\in Q\mid xa=ax \ \forall x\in Q\}$,\\
\emph{the center of Q}, &$Z(Q)=N(Q)\cap C(Q)$.\\
\end{tabular}
\end{center}
For a loop $Q$, the nuclei $N(Q),N_{\lambda}(Q),N_{\mu}(Q),$ and $N_{\rho}(Q)$ are all subloops of $Q$ and the center $Z(Q)$ is a normal subloop of $Q$.  However, the commutant, $C(Q)$ need not be a subloop in general of $Q$.

\begin{proposition}
Let $Q$ be a loop.  Then $a\in C(Q)\cap N_{\lambda}(Q) \Leftrightarrow R_{a}\in Z(Mlt_{\rho}(Q))$.
\label{commleftnuc}
\end{proposition}
\begin{proof}
Let $a\in C(Q)\cap N_{\lambda}(Q)$.  Then $\forall x,y\in Q$,
\[
yR_{a}R_{x}=ya\cdot x=ay\cdot x=a\cdot yx=yx\cdot a=yR_{x}R_{a}.
\]
Hence, $R_{a}\in Z(\Mlt_{\rho}(Q))$.  Conversely, let $R_{a}\in Z(\Mlt_{\rho}(Q))$.  Then $ax=1R_{a}R_{x}=1R_{x}R_{a}=xa$.  Hence $a\in C(Q)$.  Moreover, 
\[
a\cdot yx=yx\cdot a=yR_{x}R_{a}=yR_{a}R_{x}=ya\cdot x=ay\cdot x.
\]
Thus, $a\in C(Q)\cap N_{\lambda}(Q)$.
\end{proof}

\begin{proposition}
For a loop $Q$, the following are equivalent:
\begin{enumerate}
\item  $Q$ is an RCC loop,
\item  The following holds for all $x,y,z\in Q$:
\begin{equation}\tag{RCC$_1$} 
R_{x}^{-1}R_{y}R_{x}=R_{x\backslash yx}. \label{RCC1}
\end{equation}
\item  The following holds for all $x,y,z\in Q$:
\begin{equation}\tag{RCC$_2$} 
(xy)z=(xz)\cdot z\backslash(yz).\label{RCC2}
\end{equation}
\item  For all $x\in Q$, $(R_a,R_a L_a^{-1},R_a)\in \Atp(Q)$.
\end{enumerate}
\end{proposition}
\begin{proof}
If $Q$ is an RCC loop, then $\forall x,y\in Q$, we have $R_{x}^{-1}R_{y}R_{x}=R_{z}\Leftrightarrow R_{y}R_{x}=R_{x}R_{z}$.  Hence, applying this to $1$ gives $yx=xz$, and thus, $z=x\backslash yx$.  Similarly, \eqref{RCC1} holds if and only if $R_{y}R_{z}=R_{z}R_{z\backslash yz}$ for all $y,z \in Q$, which is clearly equivalent to \eqref{RCC2}.  Finally, $(R_a,R_a L_a^{-1},R_a)\in \Atp(Q)$ is simply \eqref{RCC2}.
\end{proof}

\begin{proposition}
Let $Q$ be a RCC loop.  Then 
\begin{itemize}
\item [(i)] $N_{\mu}(Q)=N_{\rho}(Q)\trianglelefteq Q$ and
\item [(ii)] $C(Q)\leq N_{\lambda}(Q)$.
\end{itemize}
\label{nucleuses}
\end{proposition}
\begin{proof}
For $(i)$, note that 
\[
(id_{Q},R_{a},R_{a})(R_{a},L_{a}^{-1},id_{Q})=(R_{a},R_{a}L_{a}^{-1},R_{a})\in \Atp(Q).
\]  
Therefore, if $(id_{Q},R_{a},R_{a})$ or $(R_{a},L_{a}^{-1},id_{Q})$ is in $\Atp(Q)$, the other one is as well. For normality, see \cite{drapal03}.

For $(ii)$, let $a\in C(Q)$.  Then, using \eqref{RCC2}, we have \[
ax \cdot y=xa\cdot y = xy \cdot y\backslash (ay)=xy\cdot a=a \cdot xy. \qedhere
\]
\end{proof}

% % % % % % % % % % % % % % % % % % % % % % % % % % % % % % % % % % % % % % % % % % %
% % % % % % % % % % % % % % % % % % % % % % % % % % % % % % % % % % % % % % % % % % %
% % % % % % % % % % % % % % % % % % % % % % % % % % % % % % % % % % % % % % % % % % %
\section{Constructing Simple RCC loops}
\label{S3}
$\quad$ Let $\mathbb{F}_{q}$ be the finite field of order where $q=p^{n}$ for a prime $p$ and some $n>0$.  For a matrix $M$, let $Det(M),Tr(A),$ and $Char(M)$ denote the standard \emph{determinant}, \emph{trace} and \emph{characteristic polynomial of the matrix M}.  In this paper, all matrices will be of size $2\times 2$ (\emph{i.e.} $M\in GL(2,q)$), hence $Char(M)=x^{2}-Tr(M)x+Det(M)\in \mathbb{F}_{q}[x]$.

Let $f(x) = x^{2}-rx + s$ be irreducible in $\mathbb{F}_q[x]$.  For each $b \in \mathbb{F}_q$, define 
\[
M_{(0,b)} = \begin{pmatrix} b&0\\ 0&b \end{pmatrix}
\] 
and for $a\neq 0$, 
\[
M_{(a,b)} = \begin{pmatrix} r-b& \frac{f(b)}{-a}\\a&b \end{pmatrix}.
\]
Note that $Det(M_{(a,b)})=s$ and $Tr(M_{(a,b)})=r$ and thus $Char(M_{(a,b)})=f(x)$.

\begin{lemma}
Let $f(x) = x^2-rx + s$ be irreducible in $\mathbb{F}_q[x]$.  The conjugacy class of all matrices in $GL(2,q)$ with characteristic polynomial $f(x)$ is precisely the set $\{  M_{(a,b)} \mid a,b\in \mathbb{F}_q \}$ for $a\neq 0$.
\end{lemma} 
\begin{proof}
Note that if two elements of $GL(2,q)$ are conjugate then they both have the same characteristic polynomial, and hence for a $2\times 2$ matrix, have the same determinant and trace \cite{rotman94}. Now suppose $M = \begin{pmatrix} c & d \\ a & b \end{pmatrix}$ has $Char(M) = f(x)$. Note that $a\neq 0$ since $f(x)$ is irreducible; otherwise, $M$ would have $c$ and $b$ as eigenvalues. Now $r = Tr(M) = c + b$, so that $c = r-b$. Also, $s = Det(M) = (r-b)b - da$, and so $-da = b^2 - rb + s = f(b)$.  Hence $d = f(b)/(-a)$. Therefore $M = M_{(a,b)}$ as claimed.
\end{proof}
Let $f(x) = x^{2}-rx + s$ be irreducible in $\mathbb{F}_q[x]$.  Let $Q = \mathbb{F}_q^2\backslash\{[0,0]\}$, written as a set of row vectors. Define a binary operation $\circ_f$ on $Q$ by 
\[
[a,b]\circ_f [c,d] = [a,b] M_{(c,d)}.
\]
Note that 
\begin{alignat*}{3}
[a,b]\circ_{f}[c,d]
&=[a(r-d)+bc,\frac{-af(d)}{c}+bd] \qquad
&&c\neq 0,\\
[a,b]\circ_{f}[0,d]
&=[ad,bd] 
&&c=0.
\end{alignat*}

It is clear that $\circ_{f}$ is closed on $Q$.  Indeed, if $[a,b]\circ_{f}[c,d]=[0,0]$ and $c=0$, then either both $a=b=0$ or $d=0$.  

For $c\neq 0$, if $d=0$, then $ar+bc=\frac{-as}{c}=0$.  Thus, either $a=0$ implying $b=0$ or $s=0$.  For $d\neq 0$, we have 
\[
r-d=\frac{-bc}{a}= -d+r-\frac{s}{d}
\]
implying $s=0$.  Therefore, $[a,b]\circ_{f}[c,d]=[0,0]$ \emph{if and only if} either $[a,b]=[0,0]$ or $[c,d]=[0,0]$.
\begin{remark}
To keep notation clear, 
\begin{enumerate}
\item $[x,y]$ denotes an element in $Q$; 
\item $R_{[x,y]}$ denotes the right translation by $[x,y]$; 
\item $M_{(x,y)}$ denotes the matrix associated with the right translation by $[x,y]$.
\item maps on $Q$ act of the right (\emph{i.e.} $x\theta$) and maps on $\mathbb{F}_q$ act on the left (\emph{i.e.} $\theta(x)$).
\end{enumerate}
\end{remark}
\begin{theorem} (Hall \cite{hall55})
$(Q,\circ_f)$ is a loop with identity element $[0,1]$ with $0$ and $1$ being the additive and multiplicative identity in $\mathbb{F}_q$ respectively.
\label{construction}
\end{theorem}
An isomorphic construction can be found in \cite{hall55}.  To keep this paper self-contained, we give our own proof of the following.
\begin{proof}
First note that $R_{(Q,\circ_{f})}=\{  M_{(a,b)} \mid a,b\in F_q  \}\backslash\{M_{(0,0)}\}$ by the definition of $\circ_{f}$.  That is, $R_{[a,b]}$ corresponds uniquely to $M_{(a,b)}$ by construction.  Now, by Proposition \ref{rightsectionloop}, it is enough to show that each $R_{[y,z]}R_{[u,v]}^{-1}=M_{(y,z)}M^{-1}_{(u,v)}$ is fixed-point free.

Let $M_{(y,z)},M_{(u,v)}\in R_{(Q,\circ_{f})}$ and suppose $M_{(y,z)}M_{(u,v)}^{-1}$ has a fixed point.  Then, $M_{(y,z)}M_{(u,v)}^{-1}$ has an eigenvalue of $1$.  Let $g(x)=Char(M_{(y,z)}M_{(u,v)}^{-1})$.  Then
\begin{align*}
g(x)&=x^{2}-Tr(M_{(y,z)}M_{(u,v)}^{-1})x+Det(M_{(y,z)}M_{(u,v)}^{-1}),\\
0=g(1)&=1^{2}-Tr(M_{(y,z)}M_{(u,v)}^{-1})+Det(M_{(y,z)}M_{(u,v)}^{-1})\\
&=1-Tr(M_{(y,z)}M_{(u,v)}^{-1})+1.
\end{align*}  
Thus, $Tr(M_{(y,z)}M_{(u,v)}^{-1})=2$.  Therefore, $g(x)=x^{2}-2x+1=(x-1)^{2}$.  Then, either 
$M_{(y,z)}M_{(u,v)}^{-1}=\begin{pmatrix} 1&0\\0&1 \end{pmatrix}$ or $M_{(y,z)}M_{(u,v)}^{-1}$ is similar to $\begin{pmatrix} 1&1\\0&1 \end{pmatrix}$.  In the first case, we have $M_{(y,z)}=M_{(u,v)}$.  For the second, suppose $M_{(y,z)}\neq M_{(u,v)}$ and let $P\in GL(2,q)$ such that $PM_{(y,z)}M_{(u,v)}^{-1}P^{-1}=\begin{pmatrix} 1&1\\0&1 \end{pmatrix}$.  Then define $A = PM_{(y,z)}P^{-1}$ and \\$B = PM_{(u,v)}P^{-1}$, so that $AB^{-1} = \begin{pmatrix} 1&1\\0&1 \end{pmatrix}$.  Note that $A$ and $B$ have the same determinant and trace as $M_{(y,z)}$ and $M_{(u,v)}$, respectively and hence $Char(A)=Char(B)=f(x)$.  Let $A=\begin{pmatrix} a&b\\c&d \end{pmatrix}, B=\begin{pmatrix} e&f\\g&h \end{pmatrix}$.  Then $[1, 0]A=[1 ,1]B$ and $[0, 1]A=[0,1]B$.  Hence $a=e+g,b=f+h,c=g,d=h$.  Thus $A=\begin{pmatrix} e+g& f+h\\g&h \end{pmatrix}$ and since $Tr(A)=Tr(B)$, $g=0$.  Hence, $A,B$ are upper triangular matrices and therefore $Char(A)=f(x)$ is reducible, which is a contradiction.
\end{proof}
\begin{lemma}
In $(Q,\circ_{f})$
\begin{itemize}
\item [(i)] for $a\neq 0$, $R_{[a,b]}^{-1}=M_{(a,b)}^{-1}= \\
\begin{pmatrix} r-b& \frac{f(b)}{-a}\\a&b \end{pmatrix}^{-1}=\frac{1}{s}\begin{pmatrix} b & f(b)/a \\ -a & r-b \end{pmatrix}= \frac{1}{s} M_{[-a,r-b]}$,
\item [(ii)] $R_{[0,b]}^{-1}=\frac{1}{b}\begin{pmatrix}1&0\\0&1\end{pmatrix}$,
\item [(iii)] $R_{[a,b],[c,d]}=M_{(a,b)}M_{(c,d)}M_{[a,b]\circ_{f}[c,d]}^{-1}= \\
\begin{pmatrix} s & \frac{-(a^{2}sf(d)-abcds-abcd+abcr+acdr-acr^{2}+acrs+c^{2}f(b))}{(ac(bc-ad+ar))}\\ 0 & 1 \end{pmatrix}$,
\item [(iv)] $R_{[a,b],[0,d]}=M_{(a,b)}M_{(0,d)}M_{[a,b]\circ_{f}[0,d]}^{-1}= 
\begin{pmatrix} d^{2} & \frac{(d-1)(b-r+bd)}{a}\\ 0 & 1 \end{pmatrix}$,
\item [(v)] $R_{[0,b],[c,d]}=M_{(0,b)}M_{(c,d)}M_{[0,b]\circ_{f}[c,d]}^{-1}= 
\begin{pmatrix} b^{2} & \frac{(b-1)(d-r+bd)}{c}\\ 0 & 1 \end{pmatrix}$ and
\item [(vi)] $R_{[0,b],[0,d]}=M_{(0,b)}M_{(0,d)}M_{[0,b]\circ_{f}[0,d]}^{-1}= 
\begin{pmatrix} 1 & 0\\ 0 & 1 \end{pmatrix}$.
\end{itemize}
\label{rightinnermaps}
\end{lemma}
\begin{proof}
For $(i)$, simply note
\begin{align*}
&[x,y]\begin{pmatrix} r-b& \frac{f(b)}{-a}\\a&b \end{pmatrix}\begin{pmatrix} r-b& \frac{f(b)}{-a}\\a&b \end{pmatrix}^{-1}\\
&=[x(r-b)+ay,\frac{-x(f(b))}{c}+by] \begin{pmatrix} \frac{b}{s} & \frac{f(b)}{sa} \\ \frac{-a}{s} & \frac{r-b}{s} \end{pmatrix}\\
&=[x,y].
\end{align*}

Similarly, for $(ii)$.  For $(iii)$, using $(i)$, we have 
\[
M_{(a,b)\circ_{f}(c,d)}^{-1}=
\begin{pmatrix}
\frac{\frac{-af(d)}{c}+bd}{s} & \frac{f(\frac{-af(d)}{c}+bd)}{sa}\\
\frac{-(a(r-d)+bc)}{s} & \frac{r-(\frac{-af(d)}{c}+bd)}{s}
\end{pmatrix}.
\]
Therefore, we have
\begin{align*}
&\begin{pmatrix} r-b& \frac{f(b)}{-a}\\a&b \end{pmatrix}
\begin{pmatrix} r-d& \frac{f(d)}{-c}\\c&d \end{pmatrix}
\begin{pmatrix}
\frac{\frac{-af(d)}{c}+bd}{s} & \frac{f(\frac{-af(d)}{c}+bd)}{sa}\\
\frac{-(a(r-d)+bc)}{s} & \frac{r-(\frac{-af(d)}{c}+bd)}{s}
\end{pmatrix}\\
&=\begin{pmatrix} s & \frac{-(a^{2}sf(d)-abcds-abcd+abcr+acdr-acr^{2}+acrs+c^{2}f(b))}{(ac(bc-ad+ar))}\\ 0 & 1 \end{pmatrix}.
\end{align*}
A similar calculation gives $(iv)$.  Finally, $(v)$ and $(vi)$ follow from $(iv)$ and Lemma \ref{commleftnuc}.
\end{proof}
For a loop $Q$, $x$ has a \emph{two-sided inverse} if $1/x=x\backslash 1$, denoted $x^{-1}$.   A loop is said to satisfy the \emph{right inverse property} if $(yx)x^{-1}=y\Leftrightarrow R_{x^{-1}}=R_{x}^{-1}$ for all $x,y\in Q$.  Note that a loop $Q$ satisfying RIP has two-sided inverses for all $x\in Q$.
\begin{corollary}
$(Q,\circ_f)$ satisfies RIP.
\end{corollary}
\begin{proof}
This follows quickly from \eqref{rightinnermaps} $(i)$ and $(ii)$.
\end{proof}
\begin{corollary}
$|GL(2,q)|=|(Q,\circ_{f})||\Inn_{\rho}(Q,\circ_{f})|$
\end{corollary}
\begin{proof}
Note that $|R_{(Q,\circ_{f})}|=|(Q,\circ_{f})|$.
By Lemma \ref{rightinnermaps}, $\Inn_{\rho}(Q,\circ_{f})=\{\begin{pmatrix} x&y \\ 0&1 \end{pmatrix} \mid x,y\in \mathbb{F}_{q}\}$.  Hence, for $A\in GF(2,q)$ where $A=\begin{pmatrix} a&b\\c&d \end{pmatrix}$ for some $a,b,c,d\in \mathbb{F}_{q}$, $c\neq 0$, we have $A=BC$ where 
\begin{equation*}
B=\begin{pmatrix} \frac{as}{ad-bc} & \frac{a^{2}d^{2}-ra^{2}d+sa^{2}-2abcd+rabc+b^{2}c^{2}}{bc^{2}-acd}\\ \frac{cs}{ad-bc}& -\frac{as-adr+bcr}{ad-bc}\end{pmatrix} 
\qquad 
C=\begin{pmatrix} \frac{ad-bc}{s} & \frac{ad^{2}+as-bcd-adr+bcr}{cs}\\ 0 & 1 \end{pmatrix}
\end{equation*} 
It is easy to see that $Det(B)=s$ and $Tr(B)=r$, and therefore, $B\in R_{(Q,\circ_{f})}$.  It is also clear that $C\in \Inn_{\rho}(Q,\circ_{f})$.  If $c=0$, then we have 
\[
B=
\begin{pmatrix}
d&0\\0&d
\end{pmatrix}
\qquad
C=
\begin{pmatrix}
\frac{a}{d}&\frac{b}{d}\\
0&1
\end{pmatrix}
\]
and it is easy to see that $A=\begin{pmatrix}a&b\\0&d\end{pmatrix}=BC$
\end{proof}

It is well known that the center of $GL(n,q)$ are scalar multiples of $I$ \cite{rotman94}.  Thus, we have the following:
\begin{lemma}
$C(Q,\circ_{f})=\{[0,b]\mid \forall b\in \mathbb{F}_{q}\ b\neq 0 \}$.  That is, the only elements of $C(Q,\circ_{f})$ are in the set $\{R_{[a,b]} \mid [a,b] \in C(Q,\circ_{f})\}$.  Moreover, $C(Q,\circ_f)$ is a subloop of $(Q,\circ_f)$.
\label{centerRCC}
\end{lemma}
\begin{proof}
Using Propositions \ref{commleftnuc}, \ref{nucleuses} and the above remark, we are done. 
\end{proof}

Now, the loop $(Q,\circ_{f})$ has been constructed such that $R_{(Q,\circ_{f})}$ is a union of conjugacy classes in $GL(2,q)$, namely the center $Z(GL(2,q))$ (scalar matrices) and the conjugacy class of matrices $M$ with $Char(M)=f(x)$.

\begin{theorem}
$(Q,\circ_{f})$ is an RCC loop.
\label{RCC}
\end{theorem}
\begin{proof}
Let $[a,b]\in (Q,\circ_{f})$ .  First, if $a=0$ then, $M_{(0,b)}=\begin{pmatrix} b& 0\\ 0&b \end{pmatrix}$ and 
\[
[0,b] \in C(Q,\circ)\cap N_{\lambda}(Q,\circ) \Rightarrow R_{(0, b)}\in Z(Mlt_{\rho}(Q,\circ)) 
\]
by Proposition \ref{commleftnuc} and Lemma \ref{centerRCC}.  Therefore, for any $[c,d]\in (Q,\circ)$,
\[
R_{[c,d]}R_{[0,b]}R_{[c,d]}^{-1}=M_{(c,d)}M_{(0,b)}M_{(c,d)}^{-1}=M_{(c,d)}M_{(c,d)}^{-1}M_{(0,b)}=M_{(0,b)}=R_{[0,b]}.
\]
Else, let $[c,d]\in (Q,\circ)$ and see that 
\begin{align*}
Det(M_{(c,d)}M_{(a,b)}M_{(c,d)}^{-1})&=Det(M_{(c,d)})Det(M_{(a,b)})Det(M_{(c,d)}^{-1})\\
&=sss^{-1}=s=Det(M_{(a,b)}).
\end{align*}
Similarly, $Tr(M_{(c,d)}M_{(a,b)}M_{(c,d)}^{-1})=r$.  Hence  $R_{[c,d]}R_{[a,b]}R_{[c,d]}^{-1}\in R_{(Q,\circ)}$.
\end{proof}

\begin{lemma} Let $q\neq 3$.  Then $C(Q,\circ_{f})=N_{\lambda}(Q,\circ_{f})$.  If $q=3$ and $r\neq 0$, then $C(Q,\circ_{f})= N_{\lambda}(Q,\circ_{f})$.  Finally, if $q=3$ and $r=0$, then $C(Q,\circ_f)\leq N_{\lambda}(Q,\circ_f)$.
\label{nuc}
\end{lemma}
\begin{proof}  For $q=2$, $|(Q,\circ_{f})|=3$, and is an abelian group.  Let $q>3$ and note that there exists a $d\in \mathbb{F}_{q}$ such that $d^{2}\neq 1$.  Suppose $[x,y]\in N_{\lambda}(Q,\circ_{f})$.  Then for any $a\in \mathbb{F}_{q}\backslash\{ 0\}$,
\[
([x,y]\circ_{f} [a,0])\circ_{f} [0,d] = [x,y] \circ_{f} ([a,0] \circ_{f} [0,d]),
\]
or equivalently, $[x,y]R_{[a,0],[0,d]}=[x,y]$.  Hence, by Proposition \ref{rightinnermaps}$(iv)$, $d^{2}x=x$.  But $d^{2}\neq 1$, and thus we have $x=0$.  For $r\neq 0$, let $d\neq 1$.  Then, as before, Proposition \ref{rightinnermaps}$(iv)$ gives 
\[
y-\frac{rx(d-1)}{a}=y.
\]
But $r\neq 0$ and hence, $x=0$.  When $q=3$ and $r=0$, $C(Q,\circ_{f})<N_{\lambda}(Q,\circ{f})$ \cite{GAP,GAPNV}.
\end{proof}

Our goal is to construct simple RCC loops.  Therefore, it is vital to understand the structure of normal subloops of an RCC loop.  Let $Q$ be a RCC-loop with $N\trianglelefteq Q$ and consider $R_{N}=\{R_{x}\mid x\in N\}$.  Fix $x\in N$ and then $\forall y \in Q$, $R_{y}R_{x}R_{y}^{-1}=R_{(yx/y)}\in R_{N}$ since $yx/y \in N$.  Hence, normal subloops of $Q$ correspond to unions of conjugacy classes in $R_{Q}$.  That is, normal subloops of $Q$ correspond to unions of conjugacy classes of matrices in GL(2,q) which are contained in $R_{(Q,\circ_f)}$.  $R_{(Q,\circ_f)}$ itself is the union of conjugacy classes, namely, $\{ M_{(a,b)} | a,b\in Q, a,b\neq 0 \}$, which has size $q^2 - q$, and the $q-1$ one-element conjugacy classes in the center of $GL(2,q)$. Since the order of a normal subloop of $Q$ must divide $|Q| = q^2 - 1$, we have the following.

\begin{lemma}
The only two non-trivial normal subgroups of $(Q,\circ_{f})$ are $\{[0,1], [0,-1]\}$ and $C(Q,\circ_{f})$.
\end{lemma}
\begin{proof}
Using the above remark, the only options are matrices of the form $\begin{pmatrix} b&0\\ 0&b \end{pmatrix}$.  Hence, either we have the $C(Q,\circ_{f})$ or $\{[0,1], [0,-1]\}\leq C(Q,\circ_{f})$.
\end{proof}

\begin{theorem}
Let $f(x) = x^2 - rx + s$ be irreducible. If $r\neq 0$, then $(Q,\circ_f)$ is simple. If $r = 0$, then $Z(Q,\circ_f) = \{ [0,\pm 1]\}$ and $(Q,\circ_f)/Z(Q,\circ_f)$ is simple.
\label{simple}
\end{theorem}
\begin{proof}
Let $Tr(M_{(a,b)})\neq 0$ and suppose $(N,\circ_{f})\trianglelefteq (Q,\circ_{f})$.  Then, by Lemma \ref{nuc}, $(N,\circ_{f})\leq C(Q,\circ_{f})=N_{\lambda}(Q,\circ_{f})$.  Fix $[0,z]\in (N, \circ_{f})$ and let $[0,a], [0,c] \in (Q,\circ_{f})$.  Then 
\[
[c,0] \circ_{f} ([a,0] \circ_{f} [0,z]) = ([c,0] \circ_{f} [a,0])\circ_{f} [0,z].
\]
Thus, $[cr , \frac{-cs}{az}]=[crz , \frac{-csz}{a}]$.  Hence $z=1$. That is, if 
\[
(N,\circ_{f})\trianglelefteq (Q,\circ_{f}) \Leftrightarrow (N,\circ_{f})=\{[0,1]\}.
\]
Therefore, the only normal subloops are trivial and $(Q,\circ_{f})$ is simple.

Else, let $[a,b],[c,d]\in (Q,\circ_{f})$ and $[0,z]\in (N,\circ_{f})$  Note that 
\[
M_{(a,b)}=\begin{pmatrix} -b& \frac{s+b^{2}}{-a}\\a&b \end{pmatrix} \qquad 
M_{(c,d)}=\begin{pmatrix} -d& \frac{s+d^{2}}{-c}\\c&d \end{pmatrix}.
\]
Now, 
\[
[c,d] \circ_{f} ([a,b] \circ_{f} [0,z]) = ([c,d] \circ_{f} [a,b])\circ_{f} [0,z].
\]
implies
\[
[z(ad-bc),bdz-\frac{c(b^{2}z^{2}+s)}{az}]=[z(ad-bc), z(bd-\frac{c(b^{2}+s)}{a})].
\]
This is only solvable when $z=\pm 1$, and hence, $Z(Q,\circ_{f})=\{[0,\pm 1]\}$.  Therefore, $(Q,\circ_{f})$ is not simple.  However, $(Q,\circ_{f})/Z(Q,\circ_{f})$ is simple, since our same computation would for $z=\pm 1$ in $(Q,\circ_{f})/Z(Q,\circ_{f})$, but $[0,1]=[0,-1]$ in this loop.  Thus, the only possible normal subloops are again trivial.
\end{proof}
$\quad$ The following is an example for constructing a simple RCC loop of order $8$, from $GL(2,3)$.
\begin{example}
Let $q=3$, thus elements of $(Q,\circ_{f})$ are
\[
\{1=[0,1],2=[0,2],3=[1,0],4=[1,1],5=[1,2],6=[2,0],7=[2,1],8=[2,2]\}.
\]
Take $f(x)=x^{2}+2x+2$, irreducible in $\mathbb{F}_{3}$.  The conjugacy class of all matrices in $GL(2,3)$ with characteristic polynomial $f(x)$ are
\small
\[
\left\{
\begin{pmatrix}
1&1\\1&0
\end{pmatrix},
\begin{pmatrix}
0&1\\1&1
\end{pmatrix},
\begin{pmatrix}
2&2\\1&2
\end{pmatrix},
\begin{pmatrix}
1&2\\2&0
\end{pmatrix},
\begin{pmatrix}
0&2\\2&1
\end{pmatrix},
\begin{pmatrix}
2&1\\2&2
\end{pmatrix}
\right\},
\]
\normalsize
with the full set of matrices in $R_{(Q,\circ_f)}$
\footnotesize
\[
\left\{
\begin{pmatrix}
1&0\\0&1
\end{pmatrix},
\begin{pmatrix}
2&0\\0&2
\end{pmatrix},
\begin{pmatrix}
1&1\\1&0
\end{pmatrix},
\begin{pmatrix}
0&1\\1&1
\end{pmatrix},
\begin{pmatrix}
2&2\\1&2
\end{pmatrix},
\begin{pmatrix}
1&2\\2&0
\end{pmatrix},
\begin{pmatrix}
0&2\\2&1
\end{pmatrix},
\begin{pmatrix}
2&1\\2&2
\end{pmatrix}
\right\}.
\]
\normalsize
Note
\[ M_{(0,1)}=\begin{pmatrix}1&0\\0&1\end{pmatrix},M_{(0,2)}=\begin{pmatrix}2&0\\0&2\end{pmatrix}M_{(1,0)}=\begin{pmatrix}1&1\\1&0\end{pmatrix},\ldots 
\]
Now, act on elements in $(Q,\circ_{f})$ by the matrices above, giving the permutations for $R_{(Q,\circ_{f})}$.  For example, $M_{(2,2)}= 
\begin{pmatrix}
2&1\\2&2
\end{pmatrix}$
gives the permutation $(1,8,6,5,2,4,3,7)$ since
\begin{alignat*}{3}
[0,1]\begin{pmatrix}2&1\\2&2 \end{pmatrix}=[2,2],\qquad
&[0,2]\begin{pmatrix}2&1\\2&2 \end{pmatrix}=[1,1],\qquad
&&[1,0]\begin{pmatrix}2&1\\2&2 \end{pmatrix}=[2,1],\\
[1,1]\begin{pmatrix}2&1\\2&2 \end{pmatrix}=[2,1],\qquad
&[1,2]\begin{pmatrix}2&1\\2&2 \end{pmatrix}=[0,2],\qquad
&&[2,0]\begin{pmatrix}2&1\\2&2 \end{pmatrix}=[1,2],\\
[2,1]\begin{pmatrix}2&1\\2&2 \end{pmatrix}=[0,1],\qquad
&[2,2]\begin{pmatrix}2&1\\2&2 \end{pmatrix}=[2,0].
\end{alignat*}
\normalsize
Hence, we have 
\begin{align*}
R_{(Q,\circ_{f})}=\{
&(), (1,2)(3,6)(4,8)(5,7), (1,3,4,7,2,6,8,5), (1,4,5,6,2,8,7,3),\\
&(1,5,3,8,2,7,6,4),(1,6,7,4,2,3,5,8), (1,7,8,3,2,5,4,6), \\
&(1,8,6,5,2,4,3,7)\}.
\end{align*}
Since $r\neq 0$, $(Q,\circ_{f})$ is simple and has the following multiplication table.
\begin{table}[h]
\centering
\begin{tabular}{r|rrrrrrrr}
$\circ_{f}$ & 1 & 2 & 3 & 4 & 5 & 6 & 7 & 8\\
\hline
1&  1&  2&  3&  4&  5&  6&  7&  8 \\
2&  2&  1&  6&  8&  7&  3&  5&  4 \\
3&  3&  6&  4&  1&  8&  5&  2&  7 \\
4&  4&  8&  7&  5&  1&  2&  6&  3 \\
5&  5&  7&  1&  6&  3&  8&  4&  2 \\
6&  6&  3&  8&  2&  4&  7&  1&  5 \\
7&  7&  5&  2&  3&  6&  4&  8&  1 \\
8&  8&  4&  5&  7&  2&  1&  3&  6 \\
\end{tabular}
\caption{Multiplication Table for $(Q,\circ_{f})$}
\end{table}
\end{example}
% % % % % % % % % % % % % % % % % % % % % % % % % % % % % % % % % % % % % % % % % % %
% % % % % % % % % % % % % % % % % % % % % % % % % % % % % % % % % % % % % % % % % % %
% % % % % % % % % % % % % % % % % % % % % % % % % % % % % % % % % % % % % % % % % % %
\section{Isomorphism Classes}
\label{S4}
For $\mathbb{F}_{q}$, there are $\frac{q^{2}-q}{2}$ irreducible polynomials of degree $2$ over $\mathbb{F}_{q}$ \cite{dummitfoote04}.  Hence, it is natural to assume we would create the same number of nonisomorphic RCC loops for a given $q$.  This turns out not to be the case.  For example, when $q=4$, there are $6$ irreducible polynomials over $\mathbb{F}_{4}$ and we create $6$ RCC loops associated to each polynomial.  However, only $3$ are nonisomorphic, and each simple.  For $q=8$, we have only $10$ nonisomorphic RCC loops, instead of $28$ we can construct.  The following table gives a count of RCC loops constructed from $GL(2,q)$.  Note that RCC loops of order $p$ a prime are groups \cite{drapal03}.  Our list is exhaustive for simple RCC loops up to and including order $15$ \cite{GAP,mccune09,GAPNV}.  Also, for the loops of order $24$, $10$ loops are constructed from $GL(2,5)$ and $3$ are constructed from $GL(2,7)$.
\begin{center}
\begin{table}[h]
\footnotesize
\begin{tabular}{|c|c|c|c|c|c|}
\hline
q& Order&Number of  & Number of & Number of & Exhaustive\\
& & primitive& non-isomorphic, &Simple RCC loops& \\ 
&& polynomials&nonassociative  &&\\
&&&RCC Loops&&\\\hline	
3  & 8   & 3    & 3  & 2  & $\checkmark$\\ \hline
5  & 12  & 2    & 2  & 2  & $\checkmark$\\ \hline
4  & 15  & 6    & 3  & 3  & $\checkmark$\\ \hline
5,7& 24  & 10,3 & 13 & 11 & \\ \hline
9  & 40  &  2   & 2  & 2  & \\ \hline
7  & 48  & 21   & 21 & 18 & \\ \hline
11 & 60  &  5   & 5  & 5  & \\ \hline
8  & 63  & 28   & 10 & 10 & \\ \hline
9  & 80  & 36   & 18 & 16 & \\ \hline
13 & 84  &  6   & 6  & 6  & \\ \hline
11 & 120 & 55   & 55 & 50 & \\ \hline
13 & 168 & 78   & 78 & 72 & \\ \hline
16 & 255 & 120  & 30 & 30 & \\ \hline
\end{tabular}
\caption{Table of RCC Loops}
\end{table}
\end{center}

We now describe the isomorphism classes for this construction.  It is well-known that $\alpha \in \Aut(\mathbb{F}_q) \Leftrightarrow \alpha(x) =x^{p^{i}}$ for $0\leq i \leq n$, the Frobenius automorphisms.  Note that if $f(x)=x^2-rx+s$ is irreducible, then $g(x)=x^2-r^{p^i}x+s^{p^i}$ is irreducible as well.
\begin{theorem}
Let $f(x)=x^{2}-r_{1}x+s_{1}$ and $g(x)=x^{2}-r_{2}x+s_{2}$ be irreducible in $\mathbb{F}_{q}[x]$.  Then
$\phi:(Q,\circ_{f})\rightarrow (Q,\circ_{g})$ is an isomorphism \emph{if and only if} $[a,b]\phi=[\alpha(a),\alpha(b)]$ for some $\alpha\in \Aut(\mathbb{F}_{q})$.
\label{isoauto}
\end{theorem}
\begin{proof}
Our first goal is to show $[0,b]\phi=[0,y]$ and $[a,0]\phi=[x,0]$ for some $a,b,x,y\in \mathbb{F}_q$.
 
Since $\phi$ is an isomorphism, $\phi$ maps $C(Q,\circ_f)$ to $C(Q,\circ_g)$, \emph{i.e.} $[0,b]\phi = [0,y]$ for some $b,y\in \mathbb{F}_q$.  Suppose $[a,0]\phi = [x,y]$ and $[c,0]=[x',y']$ for some $a,c,x,x',y,y'\in\mathbb{F}_q$.  Then
\[  
([a,0]\circ_f[a,0])\phi = [r_1,s_1]\phi=[r_2,yr_2-s_2]=[x,y]\circ_g[x,y]=[a,0]\phi\circ_g[a,0]\phi 
\]
\[
([c,0]\circ_f[c,0])\phi = [r_1,s_1]\phi=[r_2,y'r_2-s_2]=[x',y']\circ_g[x',y']=[c,0]\phi\circ_g[c,0]\phi 
\]
Therefore, $yr_2-s_2 = y'r_2-s_2$ implying $y=y'$.  If $y\neq 0$, consider $b\neq 1$ and suppose $[a,0]\phi =[x,y]$ for some $x\in \mathbb{F}_q$.  We have
\[ ([a,0]\circ_f[0,b])\phi = [ab,0]\phi = [u,y]\]
for some $u\in \mathbb{F}_q$.  Similarly, supposing $[0,b]\phi = [0,v]$ for some $v\in \mathbb{F}_q$,
\[[a,0]\phi\circ_g[0,b] = [x,y]\circ [0,v] = [xv,yv].\]
But this implies $v=1$, so $[0,b]\phi=[0,1]$.  Therefore, $b=1$ since $\phi$ is an isomorphism, a contradiction.  Thus $[a,0]\phi = [x,0]$ for some $x\in\mathbb{F}_q$.

We show that $[a,b]\phi=[\alpha(a),\beta(b)] = [\alpha(a),\alpha(b)]$ for some bijections $\alpha,\beta$ of $\mathbb{F}_q$.

Suppose $[a,b]\phi=[x,y]$ and $[c,b]\phi=[z,y']$ for some $a,b,c,x,y,y',z\in \mathbb{F}_q$.  Then
\[([a,b]\circ_f[a,b])\phi = [r_1,r_1b-s_1]\phi=[r_2,yr_2-s_2]=[x,y]\circ_g[x,y]=[a,b]\phi\circ_g[a,b]\phi \]
\[([c,b]\circ_f[c,b])\phi = [r_1,r_1b-s_1]\phi=[r_2,y'r_2-s_2]=[z,y']\circ_g[z,y']=[c,b]\phi\circ_g[c,b]\phi \]
Hence, $y=y'$.

On the other hand, suppose $[a,b]\phi=[x,y]$, $[a,c]\phi=[x',z]$, and $[1,0]\phi=[u,0]$ for some $a,b,c,x,x',y,z,u\in \mathbb{F}_q$.  Then
\begin{align*}
([a,b]\circ_f[1,0])\phi &= [ar_1+b,-as_1]\phi=[w,v]\\
&=[xr_2+yu,\frac{-xs_2}{u}]=[x,y]\circ_g[u,0]=[a,b]\phi\circ_g[1,0]\phi,\\
([a,c]\circ_f[1,0])\phi &= [ar_1+c,-as_1]\phi=[w',v]\\
&=[x'r_2+zu,\frac{-x's_2}{u}]=[x',z]\circ_g[u,0]=[a,c]\phi\circ_g[1,0]\phi. 
\end{align*}
Hence, $\dfrac{-xs_2}{u}=v=\dfrac{-x's_2}{u}$, thus $x=x'$.  

Thus we have shown $[a,b]\phi = [\alpha(a),\beta(b)]$.  Since $\phi$ is an isomorphism, we have $\alpha,\beta$  bijections of $\mathbb{F}_q$  Moreover, $\alpha(0)=0$ and $\beta(1)=1$.  Now, 
\begin{align*}
[\alpha(ab),\beta(0)]&=[ab,0]\phi \\
&= ([a,0]\circ_f[0,b])\phi \\ 
&= [\alpha(a),\beta(0)]\circ_g[\alpha(0),\beta(b)]\\
&=[\alpha(a),\beta(0)]\circ_g[0,\beta(b)]\\
&= [\alpha(a)\beta(b),\beta(0)\beta(b)]
\end{align*}
Hence $\beta(0)=0$ and $\alpha(a)\beta(b)=\alpha(ab)$.  But this is true for all $a,b\in \mathbb{F}_q$, thus setting $a=1$, we have $\alpha(b)=\beta(b)$ for all $b\in \mathbb{F}_q$.  Thus, $[a,b]\phi = [\alpha(a),\alpha(b)]$ with $\alpha(ab)=\alpha(a)\alpha(b)$, $\alpha(0)=0$ and $\alpha(1)=1$.  Therefore, $\alpha$ is an automorphism of the multiplication group $\mathbb{F}_q^*$, and hence must be of the form $\alpha(x)= x^k$ for some $k$ since $\mathbb{F}_q^*$ is cyclic.

Thus, 
\[([1,0]\circ_f[1,0])\phi = [r_1,s_1]\phi = [r_1^k,s_1^k],\]
\[[1,0]\phi\circ_g[1,0]\phi = [1,0]\circ_g[1,0] = [r_2,s_2]\].

Therefore $\alpha(r_1)=r_2$ and $\alpha(s_1)=s_2$.  But $g(x)=x^2-r_2x+s_2 = x^2-\alpha(r_1)(x)+\alpha(s_1)$ is irreducible, so $\alpha(x)=x^{p^i}$ a Frobenius map, that is, $\alpha\in\Aut(\mathbb{F}_q)$.

The reverse direction is obvious, since $([a,b]\circ_f[c,d])\phi = [a,b]\phi\circ_g[c,d]\phi$ where $[a,b]\phi=[\alpha(a),\alpha(b)]$ with $\alpha\in\mathbb{F}_q$.  
\end{proof}
Note that $|\Aut(\mathbb{F}_q)|=n$ for $q=p^n$, so one would expect the number of nonisomorphic RCC loops constructed to be $\frac{q^2-q}{2n}$.  However, it is often the case that $\frac{q^2-q}{2n}\notin \mathbb{N}$.  Hence, we have the following.
\begin{theorem}
Let $p$ be a prime number and $q=p^n$.  The number of nonisomorphic RCC loops constructed from $GL(2,q)$ is $\left\lfloor \frac{q^{2}-q}{2n} \right \rfloor + \left( \frac{q^{2}-q}{2} \mod n \right)$.
\label{isocount}
\end{theorem}
\begin{proof}
This follows quickly from Theorem \ref{isoauto} and the above note.
\end{proof}
Hence, we construct $\frac{p^2-p}{2}$ distinct RCC loops from $\mathbb{F}_p$, $\frac{p^4-p^2}{4}$ distinct RCC loops from $\mathbb{F}_{p^2}$, etc.  For $\mathbb{F}_8$, $\frac{8^2-8}{6}\notin \mathbb{Z}$, so we have $\left\lfloor \frac{8^{2}-8}{6} \right \rfloor = 9$ and $\left( \frac{8^{2}-8}{2} \mod 3 \right) = 1$, and thus we have $9+1=10$ nonisomorphic RCC loops.

Lastly, we consider
\begin{question} What group is $Mlt_{\rho}(Q,\circ_{f})$?
\label{q1}
\end{question}
We have $Mlt_{\rho}(Q)=\Inn_{\rho}(Q)\cdot R_{Q}$.  Indeed, for $\theta\in \Mlt_{\rho}(Q)$ set $a=1\theta$.  Then $\psi =  \theta R_a^{-1}$ fixes $1$, hence is an element of $\Inn_{\rho}(Q)$.  Therefore, $\theta=\psi R_{a}$ and since $\Inn_{\rho}(Q)\cap R_Q = \iota$, we have the factorization.
\begin{conjecture}
$\Inn_{\rho}(Q,\circ_{f})=
\{\begin{pmatrix}
x&y\\0&1
\end{pmatrix} \mid x=a^{2}s^{m}\quad a,y\in \mathbb{F}_{q}\quad m\in \mathbb{Z}\}$.
\end{conjecture}
We know $x$ must have this form from Lemma \ref{rightinnermaps}.  The question is whether we can have any value for $y\in \mathbb{F}_{q}$.  We do have the following.
\begin{lemma}
Let $H=
\{\begin{pmatrix}
x&y\\0&1
\end{pmatrix} \mid x,y\in \mathbb{F}_{q}\}.$
Then $GL(2,q)=R_{(Q,\circ_{f})}\cdot H$.
\end{lemma}
\begin{proof}
Note that  $|R_{(Q,\circ_{f})}| = q^{2} -1$  and  $| H | = q(q-1)$.  We have $|GL(2,q)| = (q^2-1)(q^2-q) = q(q+1)(q-1)^2 =  |R_Q| |H|$ . Since $R_{(Q,\circ_{f})} \cap H= \begin{pmatrix}1&0\\0&1\end{pmatrix}$, we have the desired result.
\end{proof}
\noindent
Hence, Question \ref{q1} reduces to what subgroups of $H$ can occur as $\Inn_{\rho}(Q,\circ_{f})$?

\begin{acknowledgment}
Some investigations in this paper were assisted by the finite model builder \textsc{Mace4} developed by McCune \cite{mccune09}.  Similarly, all presented examples were verified using the GAP system \cite{GAP} together with the LOOPS package \cite{GAPNV}.  
\end{acknowledgment}

%\bibliography{bibdatabase}

\end{document}